\documentclass[12pt,oneside]{amsproc}
\usepackage[T2A]{fontenc}
\usepackage[cp1251]{inputenc}
\usepackage[english,russian]{babel}
\usepackage{amssymb}
\usepackage{amsthm}
   \usepackage{verbatim}

\theoremstyle{plain}

\def\tr{\operatorname{trace}\,}               
\def\sp{\operatorname{sp}\,}

\newtheorem{defi}{Определение}[section]
\newtheorem{theor}{Теорема}[section]
\newtheorem{lem}{Лемма}[section]
\newtheorem{rem}{Замечание}[section]
\newtheorem{exam}{Пример}[section]
\newtheorem{prop}{Предложение}[section]
\newtheorem{cor}{Следствие}[section]
\theoremstyle{definition}
\title[Собственные числа ядерных операторов]
{След, детерминант и собственные числа \\
 ядерных операторов}
\author{О.И. Рейнов}
\address{Санк-Петербургский государственный университет
}
\email{orein51@mail.ru}

\thanks{%${ }^\maltese$
AMS Subject Classification 2010: 47B10, 47A75}
\thanks{${ }$ Key words:  ядерный оператор, след, детерминант, собственное число, квазинорма, тензорное произведение}

\begin{document}

\begin{abstract}
Показывается, как новые результаты в теории детерминантов и следов,
а также в теории квазинормированных тензорных произведений 
могут быть применены для получения новых теорем о распределении 
собственных чисел ядерных операторов в банаховых пространствах
и о совпадении спектральных и ядерных следов таких операторов. В качестве
примеров рассматриваются новые классы операторов — обобщенные ядерные
операторы Лоренца-Лапресте $N_{(r,s),p}.$
\end{abstract}
%We study the possibilities of factorizations of products of nuclear operators of different types through Schatten-von Neumann 
%operators in Hilbert spaces with giving some applications to eigenvalues problems.

 %%%%%%%%%%%%%%%%%%%%%%%%%%

  \maketitle

\tableofcontents
%%%%%%%%%%%%%%%
 \medskip

%\section{Введение}

%%
%%%%%%%%%%%%%%%%%%%%%%%%%%%%%%%%%%%%%%%%%%%%%%%%%%%%%%%%%%%   
\section{Введение}
 
 В 1950-х В. Б. Лидский \cite{Lid} и А. Гротендик \cite{Gr} независимо получили знаменитые формулы следа
 для некоторых классов ядерных операторов (В. Б. Лидский ~--- в гильбертовых пространствах $H,$
 А. Гротендик ~--- в общих банаховых пространствах $X$): {\it ядерный след соответствующего оператора равен
 его спектральному следу}.
 Напомним, что к классу ядерных операторов
 в $X$ принадлежат операторы $T: X\to X,$ которые допускают представления вида
\begin{equation}\nonumber
T(x)=\sum_{k=1}^\infty \lambda_k x'_k(x) x_k\ \text{ для }\ x\in X,
 \end{equation}
 гле числа $\lambda_k,$ функционалы $x'_k\in X^*$ и элементы $x\in X$ удовлетворяют некоторым условиям суммируемости
 (при этом $\sum |\lambda_k|<\infty).$

Напомним только некоторую информацию из конечномерной теории.

  Для всякого конечномерного оператора 
\begin{equation}\nonumber
T: X\to X,\ Tx=\sum_{k=1}^N x'_k\otimes x_k
  \end{equation}
ядерный след $\tr T:= \sum_{k=1}^N x'_k(x_k)$ вполне определен и не зависит
  от представления $T.$
  Также вполне определен детерминант оператора $1-T:$
\begin{equation}\nonumber
det (1-T)= \prod_j (1-\mu_j),
  \end{equation}
  где $(\mu_j)$ --- полный набор собственных чисел оператора $T.$
  В этом случае, естественно, имеем формулу следа 
\begin{equation}\nonumber
\tr T=\sum_j \mu_j.
  \end{equation}

 Для получения формулы в случае ядерных операторов надо научиться продолжать функционалы "след"
и "детерминант" с множества конечномерных операторов на соответствующие пространства ядерных операторов.
Такое продолжение, в частности, --- цель работы.
  
{\it Доказательства основных теорем о спектральных свойствах ядерных операторов
основано как раз на возможности этих продолжений}.

%%%%%%%%%%%%%%%%%%%%%% 

\section{Предварительные сведения}

Вся терминология и факты (в настоящее время, классические), приводимые здесь без каких-либо
объяснений, могут быть найдены в \cite{Goh,Gr,PiOP,PiEig}.

Пусть $X,Y$ --- банаховы пространства. Для банахова сопряженного к $X,$ мы используем обозначение $X^*.$
  Если $x\in X$ и$x'\in X^*,$ то мы используем обозначение $\langle x',x\rangle$ для $x'(x).$

 Обозначим через $X^*\widehat\otimes Y$
  пополнение тензорного произведения $X^*\otimes Y$  (рассматриваемого как линейное пространство
  всех конечномерных операторов из $X$ в $Y$) по норме
\begin{equation}\nonumber
  ||w||:= \inf \{\left(\sum_{k=1}^N ||x'_k||\, ||y_k||\right):\ w=\sum_{k=1}^N x'_k\otimes y_k\}
\end{equation}
см., например, \cite{Gr,PiEig}).
  Для $X=Y,$ естественные непрерывный линейный функционал "trace" на $X^*\otimes X$ 
имеет единственное непрерывное продолжение
 cна пространство $X^*\widehat\otimes X,$ которое мы также будем обозначать "trace".

   Положим $N(X,Y):= $ образ тензорного произведения  $X^*\widehat\otimes Y$ в пространстве $L(X,Y)$ 
   всех ограниченных линейных отображений при каноническом фактор отображении
  $X^*\widehat\otimes Y\to N(X,Y)\subset L(X,Y).$ Мы рассматриваем (гротендиковское) пространство
 $N(X,Y)$ всех ядерных операторов из
   $X$ в $Y$ с естественной нормой, индуцированной из  $X^*\widehat\otimes Y.$
 Для тензорного элемента $u\in X^*\widehat\otimes Y,$ мы обозначаем через $\widetilde{u}$ соответствующий ядерный оператор из $X$ в $Y.$
 Иногда норму проективного тензорного произведения обозначают через $\pi,$
 а инъективную норму (т. е. норму, индуцированную обычной операторной нормой --- через $\varepsilon$). Поэтому, например,
 $X^*\widehat\otimes_\pi Y= X^*\widehat\otimes Y$ и $X^*\widehat\otimes_\varepsilon Y$ --- замыкание множества конечномерных операторов в $L(X,Y).$
  
 Примеры ядерных операторов (о квазинормах см. информацию ниже). Напомним их общий вид:
  
\begin{equation}\nonumber
T(x)=\sum_{k=1}^\infty \lambda_k x'_k(x) x_k\ \text{ для }\ x\in X,
\end{equation}
  
  Например, если $0<s\le1,$ $\sum |\lambda_k|^s<\infty$ и $\{x'_k\}, \{x_k\}$ ограничены, то
  $T\in N_s(X)$ ($s$-ядерный оператор с естественной квазинормой). 
  
  Более общо, если $(\lambda_k)\in l_{s,u}, 0<u\le1$ (пространство Лоренца), то
  $T\in N_{s,u}(X)$ ($l_{s,u}$-ядерный оператор с естественной квазинормой). 
  
  Если
  $0<r\le1, 1\le p\le2,$  $(\lambda_k)\in l_r,$ т. е. $\sum |\lambda_k|^r<\infty,$ 
  $\{x'_k\}$ ограничена и $(x_k)\in l^{w}_{p'}(X)$ (см. ниже), т. е. 
  для всякого $x'\in X^*$ ряд
  $\sum |x'(x_k)|^{p'}$ сходится, то $T\in N_{r,p}(X)$ ($(r,p)$-ядерный с естественной квазинормой).

 {\it Ядерный след}\, оператора $T$ определяется как сумма ряда: 
\begin{equation}\nonumber
\tr T:= \sum \lambda_k x'_k(x_k),
\end{equation}
  {\it спектральный след}\,  оператора $T$ ~--- как сумма $\sum \mu_n,$ где $\{\mu_n\}$ ~---
  последовательность всех собственных чисел $T.$
  
  Ядерный след определен не для каждого ядерного оператора. В условиях теоремы Лидского он определен всегда,
  а в условиях теоремы Гротендика ~--- для случая, когда $\sum |\lambda_k|^{2/3}<\infty.$
  
 Напомним общий вид проективного тензорного элемента:
\begin{equation}\nonumber
z=\sum_{k=1}^\infty \lambda_k y'_k\otimes x_k\in Y^*\widehat\otimes X
\end{equation}
  Сопряженное к  $Y^*\widehat\otimes X$ пространство есть $L(X,Y^{**}).$ 
 Двойственность задается следом.
 
 Рассмотрим функционал $\widetilde T\in (Y^*\widehat\otimes X)^*,$ определяемый оператором
  $T\in L(X,Y^{**}).$ Имеем:
  
\begin{equation}\nonumber
\langle\widetilde T, z\rangle:= \tr T\circ z=
  \sum_{k=1}^\infty \lambda_k Tx_k (y'_k)
\end{equation}
 для %$z\in Y^*\widehat\otimes X.$
 $z=\sum_{k=1}^\infty \lambda_k y'_k\otimes x_k\in Y^*\widehat\otimes X.$
  
  Для $q\in (0,+\infty],$ мы обозначаем через $l_q^w(X)$ пространство всех слабо $q$-суммируемых последовательностей
  $(x_i)\subset X$ см., например, \cite{PiOP,PiEig}) с квазинормой
\begin{equation}\nonumber
   \varepsilon_q((x_i)) := \sup \{\left(\sum_i |\langle x', x_i\rangle|^q\right)^{1/q}:\ x'\in X^*,\, ||x'||\le1\}
\end{equation}
  (в случае, когда $q=\infty,$ мы предполагаем, что $(x_i)$ --- просто ограниченная, т. е.
  $\varepsilon_\infty((x_i))=\sup_i ||x_i||$).

Пространство %последовательностей
Лоренца $l_{p,q}$ $(0<p<\infty, 0\le\infty)$
состоит из последовательностей $\alpha:=(\alpha_n)\in c_0,$ для которых
\begin{equation}\nonumber
||\alpha||_{p,q}:= \left(\sum_{n\in\mathbb N} \alpha_n^{*q} n^{q/p-1}\right)^{1/q}<+\infty\
\text{при}\ q<\infty \ \text{и}
\end{equation}
\begin{equation}\nonumber
||\alpha||_{p\infty}:= \sup_{n\in\mathbb N} \alpha^*_n n^{1/p}<+\infty,
\end{equation}
где $(\alpha^*_n)$ есть неубывающая перестановка последовательности $\alpha,$
$n$-й элемент $\alpha^*_n$ которой определяется так:
\begin{equation}\nonumber
\alpha^*_n:= \inf_{|J|<n} \sup_{j\notin J} |\alpha_j|.
\end{equation}
С указанными квазинормами пространства $l_{p,q}$ являются полными квазинормированными пространствами.
При $p=q<\infty$ получаем пространства $l_p.$ 

  Еще несколько стандартных обозначений:
  $l_p(X)$ --- пространство абсолютно суммируемых последовательностей из $X,$ 
$L(X):= L(X,X),$\, $\Pi_p$ --- идеалы абсолютно $p$-суммирующих операторов, $N_s$ (для $s\in(0,1])$
--- квазинормированный идеал $s$-ядерных операторов (см. ниже более общее определения операторов
из квазинормированного идеала $N_{r.p}).$
Норма в банаховом пространстве $X$ обозначается обычно просто $||\cdot||,$ но если необходимо
подчеркнуть, в каком пространстве берется норма, то мы пишем $||\cdot||_X.$
Для последовательностей элементов некоторого множества используются обозначения типа:
$(x_k), (x_k)_k,\, (x_k)_{k=1}\infty, \{x_k\}$ и т.д.

Понятие детерминанта (Фредгольма) появится в своем месте.
Отметим только, что для элемента $u\in X^*\widehat\otimes X$ его детерминант Фредгольма
есть целая функция 
\begin{equation}\nonumber
\operatorname{det}\,(1-zu)=1-\tr u z+\dots
\end{equation}
с нулями, равными $1/\mu_k(\tilde{u}),$ --- обратным к ненулевым 
собственным значениям (каждое взятое с учетом кратности) оператора $\tilde{u}$ (см. \cite{Gr}).
%%%%%%%%%%%%%%%%%%%%%%%%%%%%%%%%%%%%%%%%%%%%%%%%%%%%%%%%%

\section{Основные определения и факты} \label{S:main1}

\subsection{Квазинормы и операторные идеалы} \label{S:QN1}
 Наше определение квазинормы несколько нестандартно.
Пусть $\alpha$ --- функция на некотором векторном пространстве $E,$ $\alpha: E\to \widehat{\mathbb R}.$ 
Мы говорим, что $\alpha$ есть {\it квазинорма} на $E,$ если
1)\,
$\alpha(E)\subset [0,+\infty]$ и
$\alpha(x)=0$ влечет $x=0;$
2)\,
 существует такая постоянная $C>0$ что
 $\alpha(x+y)\le C\, [\alpha(x)+\alpha(y)]$ для $x,y\in E;$
 3)\,
 $\alpha(ax)=|a|\, \alpha(x)$ for $a\in \mathbb K, x\in E.$
 
  \begin{defi}[]
 Пусть дана пара $(E, \alpha),$ где $\alpha$ есть квазинорма на векторном пространстве $E,$
 {\rm(i)}\
 квазинормированным пространством, ассоциированным с парой $(E, \alpha)$ называется,
квазинормированное векторное пространство
\begin{equation}\nonumber
 E_\alpha:= \left\{x\in E:\ \alpha(x)<\infty\right\}.
\end{equation}
 {\rm(ii)}\,
квазинормированным пространство $E_\alpha$ полно (= квази-банахово пространство), если
каждая последовательность Коши в $E_\alpha$\, $\alpha$-сходится к некоторому элементу из $E_\alpha.$
  \end{defi}
 
Отметим, что $E_\alpha$ является квазинормированным пространством в смысле книги \cite[p. 159]{KOT}
мы можем рассматривать соответствующую топологию (см. \cite[p. 159-160]{KOT},
  \cite[p. 445]{BenLi}).
 
 \begin{rem}
{\rm 1)\,
Вполне может быть, что выполняется равенство $E_\alpha=E.$
 
 2)\,
Хорошо известно \cite[p. 445]{BenLi},
что если $E_\alpha$ --- квазинормированное пространство,
то существует число $\beta\in (0,1]$ и $\beta$-норма $||\cdot||$
на $E_\alpha,$ эквивалентная квазинорме $\alpha.$
Напомним, что $\beta$-норма на векторном пространстве $F$
это квазинорма $||\cdot||: F\to \mathbb R,$ для которой
при всех $x,y\in F$ выполняется следующее $\beta$-неравенство треугольника:
$||x+ y||^\beta\le ||x||^\beta+||y||^\beta.$}
\end{rem}
 \smallskip
 
Напомним, что операторный идеал $\mathbb A:= (A(X,Y): X,Y \ \text{--- банаховы пространства})$
 есть подкласс класса всех линейных ограниченных операторов, компоненты $A(X,Y)\subset L(X,Y)$
 которого удовлетворяют следующим условиям:
  
 $(O_i)$\,
 $1_K\in \mathbb A,$ где $K$ обозначает одномерное банахово пространство;
 
 $(O_{ii})$\,
 Если $U,V\in A(X,Y), $ то $a_1U+a_2V\in A(X,Y)$ для всех скалярных $a_1,a_2.$
 
 $(O_{iii})$\.
 Если $S\in L(Z,X),$ $U\in A(X,Y)$ и $T\in L(Y,W),$ то $TUS\in A(Z,W.)$
 \medskip
 
 Операторный идеал $\mathbb A$ называется квазинормированным, если на нем определен класс $a$ квазинорм
 которые (обозначим их снова $a$) на компонентах являются квазинормами. обладающими свойством
 
 $(O_{iv})$\,
 $a(1_K)$=1
 
 $(O_{v})$\,
 Eсли $S\in L(Z,X),$ $U\in A(X,Y)$ и $T\in L(Y,W),$ то $a(TUS)\le ||T||\, a(U)\, ||S||.$
 
\subsection{Проективные квазинормы и свойства аппроксимации} \label{S:AP}
Теперь, пусть $\alpha$ --- квазинорма на проективном тензорном произведении
$X\widehat\otimes Y$ такая, что 
$\alpha(x\otimes y)=||x||\, ||y||$ for $x\in X, y\in Y.$
Ассоциированное квазинормированное тензорное произведение (которое мы будем обозначать 
через $X\widehat\otimes_\alpha Y$ и называть "$\alpha$-проективным тензорным произведением")
есть $\alpha$-замыкание алгебраического тензорного произведения $X\otimes Y$ в $(X\widehat\otimes Y)_\alpha$
(в конкретных случаях мы будем использовать некоторые специфические обозначения).
Таким образом,
\begin{equation}\nonumber
 X\widehat\otimes_\alpha Y:= \left\{u\in X\widehat\otimes Y:\ \alpha(u)<\infty \ \text{ and }\
    \exists\, (u_n)\subset X\otimes Y:\ \alpha(u-u_n)\underset{n\to\infty}\to 0\right\}.
\end{equation}
Более общо:
\begin{defi}[]
{\rm (i)}\,
Пусть $\widehat\otimes$ обозначает класс всех тензорных элементов проективных тензорных 
произведений произвольных банаховых пространств.
Проективная тензорная квазинорма $\alpha$ есть отображение из $\widehat\otimes$ в $\widehat{\Bbb R}$
такое, что $\alpha$ является квазинормой на каждой компоненте $X\widehat\otimes Y,$ 
обладающей свойствами:
 
$(Q_1)$\,
$\alpha(x\otimes y)=||x||\, ||y||$ для $x\in X, y\in Y.$
 
$(Q_2)$\,
Существует такая постоянная $C>0,$ что
$\alpha(u_1+u_2)\le C\, [\alpha(u_1)+\alpha(u_2)]$
для всех $X, Y$ и $u_1, u_2\in X\widehat\otimes Y.$
 
$(Q_3)$\,
Если If $u\in X\widehat\otimes Y,$\, $A\in L(X, E)$ и $B\in L(Y,F),$ то
$\alpha(A\otimes B (u))\le ||A||\, \alpha(u)\, ||B||.$
 
$(Q_4)$\,
Для всех $X, Y$ тензорное произведение $X\otimes Y$ плотно в $X\widehat\otimes_\alpha Y.$
 
{\rm (ii)}\,
Проективная тензорная норма $\alpha$ называется полной, если каждое
 $\alpha$-проективное тензорное произведение $X\widehat\otimes_\alpha Y$ является полным, то ксть
квази-банаховым.
\end{defi}
 
Для каждой проективной тензорной квазинормы $\alpha$ существует $\beta\in (0,1]$
и эквивалентная $\beta$-норма $||\cdot||_\beta$ на $\widehat\otimes$
такие, что $X\widehat\otimes_\alpha Y= X\widehat\otimes_{||\cdot||_\beta} Y$
(т. е., существует квазинорма $||\cdot||_\beta$ с $\beta$-неравенством треугольника такая, что
для некоторых положительных постоянных $C_1, C_2$ и для всех проективных тензорных
элементов $u$ выполняются неравенства $C_1 \alpha(u)\le ||u||_\beta\le C_2 \alpha(u)$). 
Таким образом, мы можем предполагать, если нужно, что a priori $\alpha$ есть $\beta$-норма.
 
Мы не будем рассматривать детально свойства введенных объектов здесь.
Однако, нам понадобится  ниже тот факт, что
отображение включения $X\widehat\otimes_\alpha Y\hookrightarrow X\widehat\otimes Y$
непрерывно для всех банаховых пространств $X, Y$
(в основных примерах \ref{E:rpq} ниже это будет автоматически выполнено).
Доказательство можно  найти в работе автора \cite[Proposition 4.1]{Trend}.
 
Так как $X\widehat\otimes_\alpha Y$ есть линейное подпространтво в $X\widehat\otimes Y,$
то пространство $L(Y,X^{*})$ разделяет точки $X\widehat\otimes_\alpha Y.$
Если $u\in X\widehat\otimes_\alpha Y,$ то $u=0$ тогда и только тогда, когда
$\operatorname{trace}\, U\circ u=0$ для
каждого $U\in L(Y,X^{*}).$ В частности, сопряженное пространство к $(X\widehat\otimes_\alpha Y)^*$
разделяет точки $X\widehat\otimes_\alpha Y.$

Ясно, что каждый тензорный элемент $u\in X\widehat\otimes_\alpha Y$ порождает
ядерный оператор $\widetilde u: X^*\to Y.$ Если $X$ является сопряженным пространством, скажем $E^*,$
то мы получаем каноническое отображение $j_\alpha: E^*\widehat\otimes_\alpha Y\to L(E, Y).$
Образ отображения $j_\alpha$ обозначается нами через $N_\alpha (E,Y),$ и снабжается
"$\alpha$-ядерной" квазинормой $\nu_{\alpha}:$ это квазинорма,
индуцированная из $E^*\widehat\otimes_\alpha Y$ фактор-отображением $E^*\widehat\otimes_\alpha Y\to N_{\alpha}(E,Y).$
Если проективная тензорная квазинорма $\alpha$ полна, то $N_\alpha(E,Y)$
является квази-банаховым пространством, а $N_\alpha$ --- квази-банахов операторный идеал.

\begin{defi}[]
Пусть $\alpha$ --- полная проективная тензорная квазинорма . Говорят, что банахово пространство
$X$ обладает свойством аппроксимации $AP_\alpha,$ если
для любого банахова пространства $E$ каноническое отображение $E^*\widehat\otimes_\alpha X\to N_\alpha(E,X)$
взаимно-однозначно (другими словами, если if $E^*\widehat\otimes_\alpha X= N_\alpha(E,X)$).
\end{defi}
 
Заметим, что если $\alpha=||\cdot||_\land,$ то мы получаем классическое свойство аппроксимации $AP$
А. Гротендика \cite{Gr}. Должно быть понятно, что $AP$ влечет the $AP_\alpha$
для любой проективной тензорной квазинормы.

Ниже нам понадобится следующая лемма.

\begin{lem} \label{L:AP}
Банахово пространство $X$ имеет свойство $AP_\alpha$ тогда и только тогда,
когда каноническое отображение $X^*\widehat\otimes_\alpha X\to L(X)$ взаимно-однозначно.
\end{lem}
 
\begin{proof}
Достаточно повторить (слово в слово с теми же обозначениями) доказательство
предложения 6.1 из \cite{Trend}.
 \end{proof}
 
\begin{exam} \label{E:rpq}
Пусть $0<r, s\le1,$ $0< p,q \le\infty$ и $1/r+1/p+1/q=1/\beta\ge1.$
Определим тензорное произведение $X\widehat\otimes_{r,p,q} Y$ как линейное подпространство
проективного тензорного произведения $X\widehat\otimes Y,$ состоящее из всех 
тензорных элементов $z,$ которые допускают представления вида
\begin{equation}\nonumber
z=\sum_{k=1}^\infty \alpha_k x_k\otimes y_k,\
(\alpha_k)\in l_r,\, (x_k)\in l_{w,p}(X),\, (y_k)\in l_{w,q}(Y);     % \label{E0}
\end{equation}
мы снабжаем его квазинормой $||z||_{r,p,q}:= \inf ||(\alpha_k)||_r\,
||(x_k)||_{w,p}\, ||(y_k)||_{w,q},$ где инфимум берется по всем представлениям
$z$ в указанной выше форме. %(\ref{E0}).
Отметим, что это тензорное произведение $\beta$-нормировано (cf. \cite{166}, где
рассмотрена версия рассматриваемого тензорного произведения как пополнение
алгебраического тензорного произведения по соответствующей "конечной" 
$||\cdot||_{r,p,q}$-квазинорме).
Оно квази-банахово (о его полноте см. препринт автора
"Approximation properties associated with quasi-normed operator ideals
of $(r,p,q)$-nuclear operators"\footnote{
http://www.mathsoc.spb.ru/preprint/2017/index.html\#08}).
Соответствующий квазинормированный операторный идеал $N_{r,p,q}$ есть квази-банахов 
идеал $(r,p,q)$-ядерных операторов (cf. \cite{PiOP, 166}).
В частных случаях, когда один или двое из показателей $p,q$
равны $\infty,$ мы используем обозначения, близкие к аналогичным обозначениям из \cite{RQ, Trend}
(но здесь мы поменяли $p', q'$ на $p, q):$
Мы обозначаем
$N_{r,\infty,\infty}$ через $N_r,$\,
$N_{r,\infty,q}$ через $N_{[r,q]},$ \,
$N_{r,p,\infty}$ через $N^{[r,p]},$\,
$\widehat\otimes_{r,\infty,\infty}$ через $\widehat\otimes_{r},$\,
$\widehat\otimes_{r,\infty,q}$ через $\widehat\otimes_{[r,q]},$\,
$\widehat\otimes_{r,p,\infty,}$ через $\widehat\otimes^{[r,p]}.$

Соответствующие обозначения используем также для свойств $AP_{r,p,q}:$

(i)\,
Для $p=q=\infty,$ мы получаем $AP_r$ из \cite{Trend}.

(ii)\,
Для $p=\infty,$ получаем $AP_{[r,q]}$ из \cite{RQ, Trend}.

(iii)\,
Для $q=\infty,$ получаем $AP^{[r,p]}$ из \cite{RQ, Trend}.
 
\end{exam}
 
Нам понадобятся некоторые факты о свойствах аппроксимации из примера \ref{E:rpq}.
Соберем их в следующей лемме.

\begin{lem} \label{L:Lp}
{\rm 1)} \cite[Corollary 10]{RQ14}\,
Пусть $s\in (0,1],$ $p\in [1,\infty]$ и $1/s=1+|1/p-1/2|.$
Если банахово пространство изоморфно подпространству фактор-пространства
(или фактор-пространству подпространства)
некоторого $L_p$-пространства, то оно имеет свойство $AP_s.$

{\rm2)} \cite[Corollary 4.1]{RQ}, \cite[Theorem 7.1]{Trend}\,
Пусть $1/r-1/p=1/2.$ Каждое банахово пространство обладает свойствами $AP_{[r,p']}$
и $AP^{[r,p']}.$
\end{lem}

Доказательство утверждения 2) может быть найдено ниже (см. пример\ref{E:main}).
См. также \cite{Trend} для других результатов в этом направлении.

\begin{rem} \label{R:AP}  {\rm
По-существу, доказательства того, что каждое банахово пространство имеет свойство $AP^{[1,2]}$
явно содержится в \cite{PiOP}. Там получено, что это утверждение
(после применения некоторых фактов из комплексного анализа)
влечет формулы типа формул Гротендика-Лидского для операторов из $N^{[1,2]}$
\cite[27.4.11]{PiOP} (и это влечет формулу Лидского для trace-класса операторов
в гильбертовых пространствах и также формулу следа Гротендика для $N_{2/3}$).
С другой стороны, существует весьма простой способ получить эти результаты
о свойствах $AP^{[1,2]}$ и $N^{[1,2]}$ из теоремы Лидского
(см. доказательства теорем \cite[Theorems 7.1-7.3]{Trend} для $p=2).$
}
\end{rem}
 
   %%%%%%%%%%%%%%%%%%%%%%  

\subsection{Факторизация через прямые суммы}. \label{S:AP2}
Ниже $X,Y$ --- произвольные банаховы пространства. 
   
 Напомним, что последовательность $(x_k)$ элементов из $X$ называется безусловным базисом, 
 если каждый $x\in X$ единственным образом разлагается в ряд $x=\sum_{k=1}^\infty a_kx_k$ и 
 этот ряд безусловно сходится (сходится при любой перестановке ряда). Это эквивалентно тому,
 что существует акая постоянная $K\ge1,$ что для любого выбора знаков $(t_k)=(\pm 1)$
 выполняется неравенство
 \begin{equation}\nonumber
 ||\sum_{k=1}^\infty t_k a_k x_k||\le K\, ||\sum_{k=1}^\infty a_k x_k||.
\end{equation}
 Безусловная константа базиса $(x_k)$ есть $ub(x_k):=\inf K.$ Таким образом, 1-безусловны базис --- это
 базис, для которого $||x||=||\sum_{k=1}^\infty t_k a_k x_k||$ для всякого $x\in X$ при любом выборе
 знаков $(t_k).$ Базис нормирован, если все его элементы имеют единичную норму.
 
 Если $(x_k)$ --- безусловный базис и $\sigma$ есть подмножество множества натуральных чисел, то естественный
 проектор $P_\sigma: X\to X,$ определяемый формулой
\begin{equation}\nonumber
 P_\sigma(x) := \sum_{k\in\sigma} a_kx_k
\end{equation}
 ограничен и $||p_\sigma||\le ub(x_k)$ (см., например, \cite[p. 18]{LinTz}).
 
 Напомним определение прямой суммы банаховых пространств. 
 Пусть $E$ --- банахово пространство с 1-безусловным нормированным базисом $(e_k)$ и $(X_i)$ ---
 последовательность банаховых пространств. {\it Прямой суммой} этих пространств по типу $E$
 называется банахово пространство $(\sum X_i)_E$ состоящее из последовательностей $(x_i),$
 $x_i\in X_i,$ для которых конечна норма 
 \begin{equation}\nonumber
 ||(x_i)||:= ||\sum_i ||x_i||\, e_i||_E.
\end{equation} 
 Пространство $(\sum X_i)_E$ обладает такими важными свойствами:
 
 $(u_1)$
 Каждое пространство $X_n$ естественным образом изометрически вкладывается в $(\sum X_i)_E$
 и его образ 1-дополняем там, т. е. существует (естественный) непрерывный проектор из $(\sum X_i)_E$ на образ $ X_n$
 и норма этого проектора равна 1. Более того, то же верно, если вместо одного пространства $X_n$
 рассмотреть конечную прямую сумму $(\sum_{i=1}^n X_i)_E$ и соответствующий проектор $P_n.$
 
 [Действительно,
 $||P_n(x_i)_{i=1}^\infty||= ||\sum_{i=1}^n ||x_i||\, e_i||_E\le ||\sum_{i=1}^\infty ||x_i||\, e_i||_E$
 по замечаниям выше.]
 
 $(u_2)$
 Если в каждой из изометрических копий пространств $X_i$ взять по элементу $x_i$
 единичной нормы, то полученная последовательность $(x_i)$ будет образовывать
 последовательность, эквивалентную базису $(e_i).$
   
 Заметим, что определить понятие прямой суммы ("по базису") с теми же хорошими свойствами 
 для пространств в базисами более слабых типов (например, условного) затруднительно
 (цитата из \cite{Kad91}: как ни определяй понятие прямой суммы бесконечномерных
 пространств $X_i$ по последовательности $(e_i),$ не являющейся
 безусловной базисной последовательностью, свойство $(u_2)$ прямой суммы не
 будет выполнено ни в каком смысле).
 
 Ниже, говоря о прямых суммах пространств, мы будем подразумевать (если не задан явно тип суммы), что 
 рассматриваемая сумма берется по типу  $E$ для некоторого пространства $E$ с 1-безусловным базисом.

Пусть $\mathbb Z:=(Z_\alpha)$ --- семейство банаховых пространств, которое с каждой
 парой пространств $Z_1, Z_2$ содержит и их прямую сумму $Z_1\oplus Z_2.$
 Обозначим через $\Gamma_Z$ совокупность всех  операторов, которые факторизуются через пространство из $\mathbb Z:$
 $T\in \Gamma_{\mathbb Z}(X,Y)$ тогда и только тогда, когда существуют пространство $Z\in \mathbb Z$ и операторы 
 $A\in L(X,Z)$ и $B\in L(Z,Y)$ такие, что
 $T=BA: X\overset A\to Z\overset B\to Y.$
 С нормой 
\begin{equation}\nonumber
 \gamma_{\mathbb Z}(T):=\inf \{||A||\, ||B||:\ \exists\, Z\in \mathbb Z, A\in L(X,Z), B\in L(Z,Y);\, T=BA\}
\end{equation}
 пространство $\Gamma_{\mathbb Z}(X,Y)$ нормировано, а $(\Gamma_{\mathbb Z},\gamma_{\mathbb Z})$ 
 есть нормированный операторный идеал.
 
 Действительно, 
 пусть $1_K$ --- тождественный оператор в одномерном пространстве $K$ и $Z\in \mathbb Z.$
 Пусть, далее, $j: K\to Z$ --- какое-либо изометрическое вложение. Продолжим отображение 
 (линейный функционал) $1_K j^{-1}: j(K)\to K$ с подпространства $j(K)\subset Z$ на все $Z$ до
 отображения $J: Z\to K$ с сохранением нормы. Ясно, что $1_K=J j: K\to Z\to K,$
 $\gamma_{\mathbb Z}(1_K)=1.$ Таким образом, выполнены условия $(O_i)$ и $(O_{iv}).$
 
 Проверим линейность (условие $(O_{ii})$). Для $U,V\in \Gamma_{\mathbb Z}(X,Y)$ пусть
 $U=B_1A_1$ и $V=B_2A_2$ --- факторизации этих операторов через пространства $Z_1$ и через $Z_2$ 
 из $\mathbb Z$ соответственно. Рассмотрим прямую сумму $Z:= Z_1\oplus Z_2,$
 обозначив через $j_k$ и $P_k$ естественные изометрические вложения $Z_k\to Z$ и проекторы
  $Z\to Z_k$ $(k=1,2)$ соответственно (так что $P_kj_k=1_{Z_k}$ и $P_1j_2=P_2j_1=0).$ Положим 
\begin{equation}\nonumber
A(\cdot):=(j_1 A_1(\cdot), j_2A_2(\cdot)): X\to Z= Z_1\oplus Z_2
\end{equation}
  и 
\begin{equation}\nonumber
 B(\cdot):= B_1P_1(\cdot)+ B_2P_2(\cdot): Z= Z_1\oplus Z_2\to Y.
\end{equation}
 Для $x\in X$ имеем:
\begin{equation}\nonumber
 BAx= B(j_1A_1x, j_2A_2x)= (B_1P_1+B_2P_2)(j_1A_1x, j_2A_2x)= B_1A_1x + B_2A_2x=Ux+Vx,
\end{equation}
  т. е., $U+V\in \Gamma_{\mathbb Z}(X,Y),$ причем ясно, что 
\begin{equation}\nonumber
  \gamma_{\mathbb Z}(U+V)\le \gamma_{\mathbb Z}(U) +\gamma_{\mathbb Z}(V).
\end{equation}
 Мультипликативность из условия $(O_{ii}$ очевидна, так же как и ясно выполнение
 условий $(O_{iii})$ и $(O_{v}).$
   
 Для полноты операторного идеала нужна сходимость соответствующих рядов. Поэтому
 мы обратимся к частному случаю рассмотренного только что идеала ("подидеалу").

 Пусть теперь $\mathbb Z:=(Z_\alpha)$ --- семейство банаховых пространств.
 замкнутое относительно взятия не более чем счетных прямых сумм
 (напомним, что надо фиксировать банахово пространство с 1-безусловным базисом $E$ и 
 говорить о прямых $E$-суммах).
 Мы рассматриваем снова $\Gamma_Z$ --- 
 идеал операторов, которые факторизуются через пространство из $\mathbb Z$ с нормой, описанной выше.
Пространство $\Gamma_{\mathbb Z}(X,Y)$ банахово, а $(\Gamma_{\mathbb Z},\gamma_{\mathbb Z})$ 
 есть банахов нормированный операторный идеал.
 Действительно, нам надо лишь установить полноту идеала. Для этого мы фиксируем $\varepsilon>o$ и рассмотрим
 сходящийся ряд
\begin{equation}\nonumber
 \sum_{k=1}^\infty \gamma_{\mathbb Z}(T_k)<\infty,
\end{equation}
 где $T_k:=B_kA_k\in \Gamma_{\mathbb Z}(X,Y),$ $A_k: X\to Z_k$ и $B_k: Z_k\to Y$
 для некоторых пространств $Z_k\in \mathbb Z.$ Будем считать, что 
\begin{equation}\nonumber
 ||A_k||\le (1+\varepsilon) \gamma_{\mathbb Z}(T_k)^{1/2}
\end{equation}
 и
\begin{equation}\nonumber
 ||B_k||\le (1+\varepsilon) \gamma_{\mathbb Z}(T_k)^{1/2}.
\end{equation}
 Покажем, что ряд $\sum T_k$ сходится в $\Gamma_{\mathbb Z}(X,Y).$
 
 Положим $Z:=(\sum Z_k)_E\in \mathbb Z.$ 
 Для каждого $k$ пусть $j_k$ и $P_k$ --- изометрическое вложение $Z_k\to Z$ и проектор
 $Z\to Z_k$ соответственно такие, что $1_{Z_k}=P_kj_k,$ $||P_k||=$ (ср. с тем,, как подобное было
 проделано выше). Определим операторы $A:X\to Z$ и  $B:Z\to Y$ равенствами 
\begin{equation}\nonumber
 A:=\sum_{k=1}^\infty j_kA_k,\ B:= \sum_{k=1}^\infty B_kP_k
\end{equation}
 Так как $\sum ||j_kA_k||\le \sum (1+\varepsilon)\gamma_{\mathbb Z}(T_k)^{1/2}$ и
 $\sum ||B_kP_k||\le \sum (1+\varepsilon)\gamma_{\mathbb Z}(T_k)^{1/2},$ то эти операторы
 вполне определены, причем
\begin{equation}\nonumber
 ||BA||=||\sum_{k=1}^\infty B_kP_kj_kA_k||=||\sum_{k=1}^\infty B_kA_k||\le
 \sum_{k=1}^\infty ||B_k||\,||A_k||\le (1+\varepsilon) \sum_{k=1}^\infty \gamma_{\mathbb Z}(T_k).
\end{equation}
 Отсюда заключаем, что $T=BA=\sum T_k,$ т. е., наш ряд сходится  в $\Gamma_{\mathbb Z}(X,Y)$
 и, следовательно, пространство $\Gamma_{\mathbb Z}(X,Y)$ полно.

%%%%%%%%%%%%%%% 
\subsection{Спектральный тип}
Пусть $T$ --- оператор в $X,$
все ненулевые собственные значения которого есть собственные числа конечной (алгебраической)
кратности и которые не имеют предельных точек кроме, быть может, нуля.
Положим $\lambda(T)= \left\{\lambda\in \text{eigenvalues}\,(T)\setminus\left\{0\right\}\right\}$
(собственные числа $T$ берутся в соответствие с их алгебраической кратностью).
Мы говорим что оператор $T\in L(X,X)$ имеет {\it спектральный тип $l_{p,q},$}
если последовательность собственных чисел $\lambda(T):= (\lambda_k(T))$
лежит в пространстве Лоренца $l_{p,q}.$ Если $T$ --- спектрального типа $l_1,$ 
то мы можем определить {\it спектральный след}\, оператора $T:$\,
$\operatorname{sp\,tr\,}(T):= \sum \lambda_k(T).$
Говорим, что подпространство $L_1(X,X)\subset L(X,X)$ --- {\it спектрального типа $l_{p,q},$}
если каждый оператор $T\in L_1(X,X)$ имеет спектральный тип $l_{p,q},$
Напомним, что операторный идеал $\mathfrak A$ имеет спектральный тип $l_{p,q},$
если каждая его компонента $\mathfrak A(X,X)$ спектрального типа $l_{p,q},$

\begin{defi}
Пусть $\alpha$ --- проективная квазинорма.
Тензорное произведение $X\widehat\otimes_\alpha X$ имеет спектральный тип $l_{p,q},$
если пространство $N_\alpha(X,Y)$ есть пространство спектрального типа $l_{p,q},$
Проективная тензорная квазинорма $\alpha$ (или тензорное произведение $\widehat\otimes_\alpha)$
--- спектрального типа $l_{p,q},$ если соответствующий операторный идеал
$N_\alpha$ имеет спектральный тип $l_{p,q},$
\end{defi}
 
\begin{exam} \label{E:sptr}
$N_1(H)$\, ($= N_{[1,2]}(H)= N^{[1,2]}(H)= S_1(H),$
(trace-класс операторов в гильбертовом пространстве) --- спектрального типа $l_1$ \cite{Weyl}.
$\widehat\otimes_{2/3}$ и $N_1\circ N_1$ --- спектрального типа $l_1$ \cite{Gr}.
$N^{[1,2]}$ --- спектрального типа $l_1$ (см. \cite[see 27.4.9, конец доказательства]{PiOP}).
$N_{[1,2]}$ --- спектрального типа $l_1$ (см. \cite[Theorem 7.2 для $p=2$]{Trend};
это следует из предыдущего утверждения. Более общо, если $1/r-1/p=1/2,$
то $\widehat\otimes_{[r,p]}= N_{[r,p]},$ $\widehat\otimes^{[r,p]}= N^{[r,p]}$ и они имеют
спектральный тип $l_1$ (см. \cite[Theorems 7.1-7.3]{Trend}; простое доказательство будет дано
ниже в примере \ref{E:main}).
\end{exam}
 
Отметим, что во всех случаях примера \ref{E:sptr} для соответствующих операторов (скажем, $T)$
верна формула следа: $\operatorname{trace}\, T=\operatorname{sp\,tr\,} T.$ Общий результат 
в этом направлении --- предложение \ref{P4}.
А вот результат для частного случая (когда рассматривается семейство всех банаховых пространств).
Он является частным случаем предложения \ref{P4}.

\begin{prop} \label{P:sptype}
Пусть $\alpha$ --- полная проективная квазинорма спектрального типа $l_1.$
Для каждого банахова пространства $X$ со свойством $AP_\alpha$ и для любого $T\in N_\alpha(X),$
имеем: $\operatorname{trace}\, T=\operatorname{sp\,tr\,} T.$ 
\end{prop}

Иногда полезно такое обращение предыдущего предложения (но для произвольной квазинормы).

\begin{prop} \label{P:sptypeno}
Пусть $\alpha$ --- полная проективная квазинорма.
Если для банахова пространства $X$ и для всякого $z\in X^*\widehat\otimes_\alpha X$
выполняется равенство $\operatorname{trace}\, z= \operatorname{sp\,tr\,} \widetilde z,$ то $X$
обладает свойством $AP_\alpha.$
\end{prop}

\begin{proof}
Предположим, что $X$ не имеет своства $AP_\alpha.$ По лемме \ref{L:AP},
найдется такой элемент $z\in X^*\widehat\otimes_\alpha X,$ что
$\operatorname{trace}\, z=1$ and $\widetilde z=0.$ По предположению,
$\operatorname{sp\,tr\,} \widetilde z=\operatorname{trace}\, z=1.$
Противоречие.
\end{proof}
 
\begin{exam} \label{E:main}
Пусть $0<r\le1,$ $1\le p\le2,$ $1/r=1/2+1/p.$

1)\,
Если $T\in N_{[r,p']}(X)$ (см. пример \ref{E:rpq}),
то $T$ допускает факторизацию
\begin{equation}\nonumber
T=BA: X\overset{A}\to l_p\overset{B}\to X,\ \, A\in N_r(X, l_p),\, B\in L(l_p, X).
\end{equation}
Полные системы собственных чисел операторов $T=BA$ and $AB$ совпадают.
Но $AB\in N_r(l_p,l_p).$ Следовательно, $AB$ (и, значит, $T)$ --- спектрального типа $l_1,$
как и всякий $r$-ядерный оператор в $l_p$ \cite[Theorem 7]{Kon}.
Отсюда вытекает, что $N_{[r,p']}$ --- спектрального типа $l_1.$
Легко видеть, что если $z\in X^*\widehat\otimes_{[r,p']} X$ таков, что $\widetilde z=T,$
то $\operatorname{trace}\, z=\operatorname{trace}\, AB$ (напомним, что $l_p$ имеет свойство $AP).$
Но $\operatorname{trace}\, AB=\operatorname{sp\,tr\,} AB$
(это установлено, например, в \cite{RQ1, Trend}, а также следует из
предложения \ref{P:sptype}). Следовательно, для каждого $z\in X^*\widehat\otimes_{[r,p']} X$
имеем: $\operatorname{trace}\, z= \operatorname{sp\,tr\,} \widetilde z.$
По предложению \ref{P:sptypeno},
каждое банахово пространство обладает свойством $AP_{[r,p']}$ $( = AP_{r,\infty,p'},$\,
см. пример \ref{E:rpq}; таким образом, мы дали и доказательство леммы \ref{L:Lp}, 2)
для случая $AP_{[r,p']}).$

2)\,
Если $T\in N^{[r,p']}(X)$ (см. пример \ref{E:rpq}),
то $T$ допускает факторизацию
\begin{equation}\nonumber
T=BA: X\overset{A}\to l_p\overset{B}\to X,\ \, A\in L(X, l_p),\, B\in N_r(l_p, X).
\end{equation}
Как и в 1), мы видим, что для любого $z\in X^*\widehat\otimes^{[r,p']} X$
имеем: $\operatorname{trace}\, z= \operatorname{sp\,tr\,} \widetilde z.$ Далее,
по предложению \ref{P:sptypeno},
каждое банахово пространство имеет свойство $AP^{[r,p]}$ $( = AP^{r,\infty,p'},$\,
см. пример \ref{E:rpq}; таким образом, мы доказали лемму \ref{L:Lp}, 2) для
случая свойств $AP^{[r,p']}).$
\end{exam}
 
Ниже нам понадобится основной результат из работы \cite{Whit}:
  
  $(W)$\, {\it Если $J$ --- квази-банахов операторный идеал спектрального типа $l_1,$ то
  спектральная сумма является следом на этом идеале $J$}.
  
  Напомним (см. определение 2.1 в \cite{Whit}), что  {\it след}\ на операторном идеале $J$ ---
  это класс комплексно-значных функций $\tau,$ каждая из которых (обозначаем снова $\tau$) задана на компоненте
  $J(E,E),$ где $E$ --- произвольное банахово пространство, такой, что
   
   (i)\  $\tau(e'\otimes e)= \langle e',e\rangle$ для всех $e'\in E^*, e\in E;$
   
   (ii)\ $\tau(AU)=\tau(UA)$ для всех банаховых пространств $E, F$ и операторов $U\in J(E,F) and    A\in L(F,E); $
   
   (iii)\ $\tau(S+U)=\tau(S) +\tau(U)$ для всех $S,U\in J(E,E);$
   
   (iv)\ $\tau(\lambda U)= \lambda \tau(U)$ для всех $\lambda\in \Bbb C$ и $U\in J(E,E).$
  
  %%%%%% $$ --> eq НИЖЕ ОК!!!
 
\subsection{Свойства $\alpha$-продолжения и $\alpha$-лифтинга}
Следующие определения и предложения понадобятся ниже. Впрочем, они 
представляют и самостоятельный интерес.

\begin{defi} \label{D:ext}
Пусть $\alpha$ --- полная проективная тензорная квазинорма.
Банахово пространство $X$ имеет свойство $\alpha$-продолжения,
если для любого подпространства $X_0\subset X$ и для всякого тензорного элемента
$z_0\in X_0^*\widehat\otimes_\alpha X_0$ существует продолжение $z\in X^*\widehat\otimes_\alpha X_0$
(так что $z\circ i= z_0$ и $\operatorname{trace}\, i\circ z= \operatorname{trace}\, z_0,$ где 
$i: X_0\to X$ --- естественное вложение).
Банахово пространство $X$ имеет свойство $\alpha$-лифтинга, если для всякого подпространства
$X_0\subset X$ и для каждого тензорного элемента $z_0\in (X/X_0)^*\widehat\otimes_\alpha X/X_0$ 
существует лифтинг
$z\in (X/X_0)^*\widehat\otimes_\alpha X$ (так что $Q\circ z=z_0,$ где $Q$ --- фактор-отображение
из $X$ на $X/X_0,$ и $\operatorname{trace}\, z\circ Q= \operatorname{trace}\, z_0).$
\end{defi}
 
\begin{rem} \label{E:remext}
Если $X$ имеет свойство $\alpha$-продолжения, то и каждое его подпространство 
имеет свойство $\alpha$-продолжения. Если пространство $X$ имеет свойство $\alpha$-лифтинга,
то и каждое его фактор-пространство имеет свойство $\alpha$-лифтинга.
\end{rem} 

\begin{exam} \label{E3.4}
Каждое банахово пространство обладает свойствами $||\cdot||_{r,\infty,q}$-продолжения и
$||\cdot||_{r,p,\infty}$-лифтинга (см. пример \ref{E:rpq}).
Для тензорных произведений $(\widehat\otimes_s, ||\cdot||_{s,\infty,\infty}),$ $s\in (0,1],$
все банаховы пространства имеют как свойство $||\cdot||_{s,\infty,\infty}$-продолжения так и
свойство $||\cdot||_{s,\infty,\infty}$-лифтинга.
Это следует из теорема Хана-Банаха и из определения банаховых фактор-пространств.
\end{exam}

%%%%%%%%%%  О ПОДППРОСТР. И Ф-ПР. НО ДЛЯ L_p %%%%%%%%%%%%%% 

\begin{theor} \label{T5.1}
Пусть $\alpha$ --- полная проективная тензорная квазинорма и
банахово пространство $X$ имеет свойство $\alpha$-продолжения.
Если $N_\alpha(X)$ имеет спектральный тип $l_{p,q},$ то всякое его подпространство также имеет
спектральный тип $l_{p,q}.$
\end{theor}

{\it Proof}.\
Пусть $X_0$ ---подпространство в $X$
и $T\in N_\alpha(X_0,X_0).$ Найдется элемент
$z_0\in X_0^*\widehat\otimes_\alpha X_0,$ для которого $\tilde z_0=T.$
По предположению, существует продолжение $z\in X^*\widehat\otimes_\alpha X_0$
(так что $z\circ i= z_0$ и $\operatorname{trace}\, i\circ z= \operatorname{trace}\, z_0,$ где 
$i: X_0\to X$ --- естественное вложение). Рассмотрим диаграмму
\begin{equation}\nonumber
i\tilde z i: X_0\overset{i}\to X\overset{\tilde z}\to X_0\overset{i}\to X.
\end{equation}
Так как $T=\tilde z_0=\tilde z i$ и $\tilde z\in N_\alpha(X,X_0),$ то $i\tilde z\in N_\alpha(X)$ и
спектр $\sp i\tilde z\setminus \{0\}\in l_{p,q}.$ 
Но собственные числа оператора $T$ (с учетом кратностей)  те же, что и собственные числа
оператора $i\tilde z.$ Следовательно, $T$  --- спектрального типа $l_{p,q.}$

\begin{theor} \label{T5.2}
Пусть $\alpha$ --- полная проективная тензорная квазинорма и
банахово пространство $X$ имеет свойство $\alpha$-лифтинга.
Если $N_\alpha(X)$ имеет спектральный тип $l_{p,q},$ то всякое его фактор-пространство также имеет
спектральный тип $l_{p,q}.$
\end{theor}

\begin{proof}
Возьмем подпространство $X_0\subset X$ и рассмотрим фак\-тор-пространство $X/X_0.$
Если $T\in N_\alpha(X/X_0, X/X_0),$ то
найдется элемент $z_0\in (X/X_0)^*\widehat\otimes_\alpha X/X_0)$ такой, что $\tilde z_0=T.$
По предположению, существует тензорный элемент $z\in (X/X_0)^*\widehat\otimes_\alpha X,$
для которого $Q\circ z=z_0,$ где $Q$ --- фактор-отображение из $X$ на $X/X_0.$
Рассмотрим диаграмму
\begin{equation}\nonumber
Q\tilde z Q: X\overset Q\to X/X_0\overset{\tilde z}\to X\overset Q\to X/X_0.
\end{equation}
Так как $T=\tilde z_0=Q \tilde z$ и $\tilde z\in N_\alpha(X,X_0),$ то $\tilde z Q\in N_\alpha(X)$ и
спектр $\sp \tilde z Q\setminus \{0\}\in l_{p,q}.$ 
Но собственные числа оператора $T$ (с учетом кратностей)  те же, что и собственные числа
оператора $\tilde z Q.$ Следовательно, $T$  --- спектрального типа $l_{p,q.}$
\end{proof}

\section{Детерминант и след} \label{S:DT}

Нам понадобятся некоторые вспомогательные факты из теории следов и детерминантов.
Ниже мы доказываем два из них; нам не удалось найти доказательства этих утверждений (именно в том виде, в котором мы хи применяем)
 в литературе, а ссылаться на очень общие аналогичные по виду теоремы, отметив, что
 "доказательства проводятся по той же схеме", не очень хорошо. Бывают случаи, когда как раз
 "общая схема" и не работает.
 Итак два предложения о непрерывности следа и о непрерывности детерминанта.
 
Напомним еще раз, что для всякого конечномерного оператора 
\begin{equation}\nonumber
T: X\to X,\ Tx=\sum_{k=1}^N x'_k\otimes x_k
\end{equation} 
  ядерный след $\tr T:= \sum_{k=1}^N x'_k(x_k)$ вполне определен и не зависит
  от представления $T.$
  Также вполне определен детерминант оператора $1-T:$
\begin{equation}\nonumber
det (1-T)= \prod_j (1-\mu_j),
\end{equation}
  где $(\mu_j)$ --- полный набор собственных чисел оператора $T.$
  В этом случае, естественно, имеем формулу следа.
\begin{equation}\nonumber
\tr T=\sum_j \mu_j.
\end{equation}
  
\begin{prop} \label{P3.1}
 Пусть $A$ --- квазинормированный операторный идеал, 
 $X$ --- банахово пространство,
 для которого множество конечномерных операторов плотно в пространстве $A(X).$ Предположим,
 что стандартный функционал $\operatorname{trace}\,$ ограничен на подпространстве всех конечномерных операторов
 из $A(X)$ (и, таким образом, может быть продолжен до непрерывного следа на все пространство $A(X)$).
 Тогда соответствующий детерминант Фредгольма равномерно непрерывен (по $A$-квазинорме) на
 некотором $A$-шаре подпространства всех конечномерных операторов из $A(X).$ Более того,
 существуют такие постоянные $r_0\in(0,1)$ и $c_0>0,$ что для конечномерных $u,v\in A(X),$
 если $||u||_A\le r_0$ и $||v||_A\le r_0,$ то
\begin{equation}\nonumber
  | \operatorname{det}\,(1-u)-\operatorname{det}\,(1-v)| \le  c_0\, ||u-v||_A.
\end{equation}
\end{prop}

\begin{proof}
Без ограничения общности, мы можем (и будем) предполагать,
 что данная квазинорма в $A$ является $s$-нормой, т. е. существует такое число $s\in (0,1],$ что
 для любых $x,y\in A$ выполняется неравенство $||x+y||_A^s\le ||x||^s_A+||y||^s_A$ (см.
 \cite[p. 1102]{Kalt}). 
 
 Обозначим через $b$ такую постоянную, что $|\operatorname{trace}\, R|\le b||R||_A$ для любого
 конечномерного оператора $R$ из $A.$
Пусть $u,v$ --- два конечномерных оператора из $A$ такие что $||u||^s_A\le r$ и $||v||^s_A\le r,$
где $r>0$ мало. Тогда (см.,  например, теорему I.3.3 в \cite{Goh} или \cite{Gr})   для $|z|\le1$
\begin{equation}\nonumber
 \operatorname{det}\, (1-zu)= \operatorname{exp} \left(-\sum_{n=1}^\infty \frac1n \operatorname{trace}\,(u^n)\, z^n\right),
\end{equation}
откуда, для малых $r>0,$ 
\begin{equation}\nonumber
| \operatorname{det}\,(1-u)-\operatorname{det}\,(1-v)| = |\operatorname{exp}\left(\sum_{n=1}^\infty 
(-\frac1n) \operatorname{trace}\,(u^n)- \sum_{n=1}^\infty (-\frac1n) \operatorname{trace}\,(v^n)\right)|\le
\end{equation}
\begin{equation}\nonumber
 \le c_1 \sum_{n=1}^\infty \frac1n |\operatorname{trace}\,(u^n)-\operatorname{trace}\,(v^n)|\le c_1 b 
 \sum_{n=1}^\infty \frac1n ||u^n-v^n||_A,
\end{equation}
где $c_1$ --- некоторая постоянная.
Если $q:=\max \{||u||_A; ||v||_A\},$ то
\begin{equation}\nonumber
 ||u^n-v^n||^s_A \le ||(u^{n-1}-v^{n-1})u||^s_A + ||v^{n-1}(u-v)||^s_A\le
\end{equation}
\begin{equation}\nonumber
 \le ||u||^s_A\, [||(u^{n-2}-v^{n-2})u||^s_A+ ||v^{n-2}(u-v)||^s_A] + ||v^{n-1}||^s_A\, ||u-v||^s_A
\end{equation}
(мы воспользовались тем, что в $A$ для любых $H,K$ имеет место соотношение $||HK||_A\le ||H||_A\,||K||_A,$
так как $||K||_L\le ||K||_A$); продолжаем неравенства:
\begin{equation}\nonumber
 \le \left(q^{(n-1)s} ||u-v||^s_A+ q^s\cdot q^{s(n-2)} ||u-v||^s_A\right) + q^{2s} ||u^{n-2}-v^{n-2}||\le
\end{equation}
\begin{equation}\nonumber
 \le 2q^{s(n-1)} ||u-v||^s_A + q^{2s} [||(u^(n-3)-v^{n-3})u||^s_A + ||v^{n-3} (u-v)||^s_A]\le
\end{equation}
\begin{equation}\nonumber
 3q^{s(n-1)} ||u-v||^s_A + q^{3s} ||u^{n-3}-v^{n-3}||^s_A\le \cdots 
\end{equation}
\begin{equation}\nonumber
 \le (n-1) q^{s(n-1)} ||u-v||^s_A + q^{s(n-1)} ||u-v||^s_A=
 n q^{s(n-1)} ||u-v||^s_A.
\end{equation}
Таким образом,
\begin{equation}\nonumber
 ||u^n-v^n||_A \le  n^{1/s} q^{(n-1)} ||u-v||_A.
\end{equation}
Поэтому, если $r\in(0,1)$ достаточно мало, и если $||u||^s_A\le r$ и $||v||^s_A\le r,$
то 
\begin{equation}\nonumber
 | \operatorname{det}\,(1-u)-\operatorname{det}\,(1-v) \le c_1 b \sum_{n=1}^\infty \frac1n ||u^n-v^n||_A \le 
\end{equation}
\begin{equation}\nonumber
 c_1 b \sum_{n=1}^\infty {n^{1/s-1}}\,{r^{n-1}} ||u-v||_A=
 c_0\, ||u-v||_A.
\end{equation}
\end{proof}

\begin{cor} \label{C3.1}
В условиях предложения 3.1, функция $\operatorname{det}\, (1-u)$ допускает непрерывное продолжение
(по $A$-квазинорме) с подпространства всех конечномерных операторов из $A(X)$
на все пространство $A(X).$
\end{cor}
\begin{proof}
Это вытекает из равномерной непрерывности 
 (по $A$-квазинорме) на  некотором $A$-шаре подпространства всех конечномерных операторов из $A(X).$ 
 Соответствующее доказательство предоставляется читателю в качестве упражнения (см, однако, \cite[p.28]{Goh}).
 \end{proof}

\begin{prop}
 Пусть $A$ --- квазинормированный операторный идеал %$(s\in(0,1]),$ 
    $X$ --- банахово пространство,
    для которого множество конечномерных операторов плотно в пространстве $A(X).$ Предположим,
    что стандартный функционал $\operatorname{det}\,(1+u)$ допускает непрерывное продолжение с подпространства всех конечномерных операторов
    из $A(X)$  на все $A(X)$  (по квазинорме из $A(X)$).
    Тогда соответствующий функционал $\operatorname{trace}\,$ ограничен (по $A$-квазинорме) на
    подпространстве всех конечномерных операторов из $A(X)$ и, таким образом,
    продолжается по непрерывности (единственным способом) на все $A(X).$
\end{prop}
   
\begin{proof}
Для конечномерного оператора $u\in A(X),$ \, $\operatorname{det}\,(1+zu)$ имеет вид
\begin{equation}\nonumber
  \operatorname{det}\,(1+zu)=1+z\operatorname{trace}\, u+ \sum_{n=1}^m a_n z^n.
\end{equation}
 Следовательно, по теореме о вычетах,
\begin{equation}\nonumber
  \operatorname{trace}\, u= \frac{1}{2\pi i} \int_{|z|=1} \frac{\operatorname{det}\,(1+zu)-1}{z^2}\, dz.
\end{equation}
 Так как $\operatorname{det}\,(1+zu)$ непрерывен в точке $u=0$ (по квазинорме из $A),$
 то существует $\delta>0,$ что $|\operatorname{det}\,(1+zu)-1|<1$ для  $||u||_A<\delta$ и $|z|\le1;$
 поэтому для таких конечномерных $u$ имеем:
\begin{equation}\nonumber
  |\operatorname{trace}\, u|\le \frac1\pi \int_{|z|=1} \|\frac{\operatorname{det}\,(1+zu)-1}{z^2}\|\, |dz|\le1.
 \end{equation}
\end{proof}   
\begin{rem}
Для доказательства предложения достаточно непрерывности детерминанта в нуле.
\end{rem}  
 
 \section{Спектральный тип и формула следа} 
  
Доказательство следующего факта проводится по аналогии с принципом равномерной ограниченности \cite[3.4.6]{PiEig}.
В отличие от теоремы из \cite{PiEig} мы рассматриваем выделенное семейство банаховых пространств, а не все банаховы пространства.
Это дает нам возможность, например, применять подобный принцип к семействам всех
$L_p(\mu)$-пространств (в качеств пространства с 1-безусловным базисом берется тогда пространство $l_p).$
  
\begin{prop} \label{P3}
Пусть $t,u>0,$ $\alpha$  --- проективная тензорная квазинорма,
  $\mathcal F$ --- некоторое семейство банаховых пространств,
  замкнутое относительно взятия не более чем счетных прямых сумм.
  Если для любого пространства $X\in\mathcal F$ пространство $N_\alpha(X)$ имеет спектральный тип $l_{t,u},$
  то существует такая постоянная $C>0,$ что для всякого $X\in\mathcal F$
  и для любого оператора $T\in N_\alpha(X)$
\begin{equation}\nonumber
||\{\mu_k(T)\}||_{l_{t,u}}\le C||T||_{N_\alpha}
\end{equation}
(здесь $\{\mu_k(T)\}$ --- полный набор собственных значений оператора $T).$
\end{prop}
\begin{proof}
Предположим противное. Тогда для каждого $n$ можно найти такие банахово пространство $X_n\in\mathcal F$
  и оператор $T_n\in N_\alpha(X_n),$ что[PiEig]
\begin{equation}\nonumber
||\{\mu_k(T_n)\}||_{l_{t,u}}\ge n \ \text{и } \    ||T_n||_{N_\alpha}\le (2\nu_\alpha)^{-n}
\end{equation}
  где $\nu_\alpha$ --- постоянная из "неравенства треугольника"  для квазинормы из $N_\alpha.$
  Положим $X:=\left(\sum_{n=1}^\infty X_n\right)_{E},$ и пусть $j_n: X\to X_n,$ $i_n: X_n\to X$ ---
  естественные фактор-отображения и вложения (с единичными нормами).
  Тогда
\begin{equation}\nonumber
  ||\sum_{n=m+1}^{m+l} j_nT_ni_n||_{N_\alpha}\le \sum_{k=1}^\infty \nu_\alpha^k\, ||T_{m+k}||_{N_\alpha}\le (2\nu_\alpha)^{-m}
\end{equation}
  для $l>0.$
  Поэтому $T:= \sum_{n=1}^\infty j_nT_ni_n\in N_\alpha(X).$
  Поскольку $T_n=j_nTi_n,$ то совокупность всех собственных чисел оператора $T_n$ есть часть
  семейства $\{\mu_k(T)\}.$ Из этого вытекает, что
\begin{equation}\nonumber
  \infty>||\{\mu_k(T)\}||_{l_{t,u}}\ge   ||\{\mu_k(T_n)\}||_{l_{t,u}}\ge n \ \text{ для  }\  n=1,2,\dots.
\end{equation}
  Полученное противоречие доказывает предложение.
\end{proof}
  
\begin{prop} \label{P4}
Пусть $r\in (0,1],$ $\alpha$  --- проективная тензорная квазинорма,
$\mathcal F$ --- некоторое семейство банаховых пространств, обладающих свойством $AP_\alpha,$ 
замкнутое относительно взятия не более чем счетных прямых сумм.
Если для любого пространства $X\in\mathcal F$ пространство $N_\alpha(X)$ имеет спектральный тип $l_r,$
то для всякого $X\in\mathcal F$
и для любого оператора $T\in N_\alpha(X)$ его ядерный след $\operatorname{trace}\, T$ вполне определен и
совпадает с его спектральным следом, т. е.
\begin{equation}\nonumber
\operatorname{trace}\, T=\sum_{k=1}^\infty \mu_k(T)
\end{equation}
(здесь $\{\mu_k(T)\}$ - полный набор, с учетом кратностей, собственных значений оператора $T).$
При этом, детерминант Фредгольма оператора $T$ имеет вид
\begin{equation}\nonumber
\operatorname{det}\,(1-zT)=\prod_{k=1}^\infty (1-\mu_k(T)z)
\end{equation}
и является целой функцией порядка $r$ (и, следовательно, минимального рода, если $r<1).$
\end{prop}
\begin{proof}
Пусть $T\in N_\alpha(X),$ где  $X\in\mathcal F.$ Так как $X\in AP_\alpha,$ то $N_\alpha(X)=X^*\widehat\otimes_\alpha X,$ что
гарантирует существование единственного непрерывного следа на $N(X),$ который есть
просто непрерывное продолжение с подпространства всех конечномерных операторов в $X$
обычного функционала "след".
По следствию \ref{C3.1} из предложения \ref{P3.1}, на $N_\alpha(X)$ вполне определен единственный непрерывный детерминант
(Фредгольма), --- $\operatorname{det}\,(1-zT).$
Возьмем последовательность $\{T_n\}$ конечномерных операторов из $N_\alpha(X),$ сходящуюся
в  пространстве $N_\alpha(X)$  к $T.$

Пространство $N_\alpha(X)$ имеет спектральный тип $l_r,$ так  что, по предложение \ref{P3},
существует такая постоянная $C>0,$ что
для любого оператора $T\in N_\alpha(X)$
\begin{equation}\nonumber
||\{\mu_k(T)\}||_{l_1}\le C||T||_{N_\alpha},
\end{equation}
в частности, это неравенство верно для всех рассматриваемых операторов.
Для конечномерного $U\in N_\alpha(X)$ детерминант имеет вид
\begin{equation}\nonumber
\operatorname{det}\,(1-zU)=\prod_{i=1}^M (1-\mu_i(U)z).
\end{equation}
Отсюда, для всякого $T_n$
\begin{equation}\nonumber
|\operatorname{det}\,(1-zT_n)|\le \operatorname{exp}\{\sum_k |\mu_k(T_n)|\, |z|\} \le e^{C||T_n||_{N_\alpha}\, |z|}.
\end{equation}
  Используя непрерывность детерминанта, мы приходим к неравенству
\begin{equation}\nonumber
   |\operatorname{det}\,(1-zT)|\le e^{C||T||_{N_\alpha}\, |z|}.
\end{equation}

По теореме Адамара,
\begin{equation}\nonumber
\operatorname{det}\,(1-zT)=e^{cz} \prod_{i=1}^\infty (1-\mu_i(T)z) e^{\mu_i(T)z}
\end{equation}
(так как значение левой части в нуле есть 1). С другой стороны, разлагая правую часть 
равенства в ряд, получаем $\operatorname{det}\,(1-zT)=1+cz+\dots.$ Значит, $c=-\tr T$
(напомним, что $\operatorname{det}\,(1-zT)=1-\tr T z+\dots.)$
Но $\{\mu_k(T)\} \in l_1$ и, следовательно.
\begin{equation}\nonumber
\operatorname{det}\,(1-zT)=e^{az} \prod_{i=1}^\infty (1-\mu_i(T)z),
\end{equation}
где $a=-\tr T+\sum \mu_i.$ 

Теперь мы применим теорему Уайта (см. \cite{Whit}). Для этого рассмотрим банахов идеал $\Gamma_{\mathcal F}$
операторов, факторизующихся через пространства из $\mathcal F$ и образуем квазинормированный.
операторный идеал $\Gamma_{\mathcal F}\circ N_\alpha$ --- суперпозицию двух идеалов.
Ясно, что этот идеал имеет спектральный тип $l_1.$ следовательно, к нему может быть применена
теорема Уайта. Поскольку идеал конечно-мерных операторов плотен в последнем
идеале, спектральный след на нем есть линейный непрерывный функционал, совпадающий с ядерным
следом на плотном множестве. 
Поэтому $a=0,$\, $\operatorname{trace}\, T=\sum_{i=1}^\infty \mu_i(T)$ и
\begin{equation}\nonumber
\operatorname{det}\,(1-zT)=\prod_{k=1}^\infty (1-\mu_k(T)z).
\end{equation}
Порядок этой целой функции есть $r,$ поскольку $(\mu_k(T))\in l_r$ (теорема Бореля 
о порядке канонического произведения).
\end{proof}

\section{Примеры применения} 
Теперь применим полученные выше вспомогательные факты в некоторых конкретных
ситуациях.     
     
\subsection{Операторы в подпространствах факторпространств $L_p$} 
Сначала рассмотрим случай ядерных операторов в подпространствах факторпространств пространств $L_p(\mu).$
 Известно, что такие пространства обладают свойством аппроксимации $AP_s$ при
 $1\le p\le \infty$ и $0<s<1,$ $1/s=1+|1/p-1/2|$ (Reinov-Latif, 2013, 2014).
 Используя этот факт и некоторые идеи (как оказалось, те же, что и в \cite[2.b.13]{Kon86}) 
 из теории абсолютно суммирующих операторов, Рейнов и Латиф сначала получили
 формулу Гротендика-Лидского для подпространств пространств $L_p$ (в работе \cite{RQ1}),
 а затем и для подпространств фактор-пространств  пространств $L_p$ (см. \cite{RQ14}).
 
 В \cite[2.c.9]{Kon86}, однако, получены более сильные результаты о спектрах ядерных операторов в
 $L_p$ (но не формула следа). Мы применим приведенные выше теоремы и предложения
 вместе с результатом из \cite[2.c.9]{Kon86} для
 установления более общих фактов, а также снова все той же формулы следа для операторов в
 подпространствах фактор-пространств  пространств $L_p.$
На прежде мы приведем утверждение, усиливающее указанный выше факт о наличии
свойств $AP_s$ в таких пространствах (что представляет и самостоятельный интерес).
Нам понадобится такая простая 
 
\begin{lem} \label{Ls1}
Пусть $0<s<1, 1/s=1+1/q.$ Если $d:=(d_k)\in l_{(s,1)},$ то найдутся 
$\alpha:= (\alpha_k)\in l_1$ и $\beta:= (\beta_k)\in l^0_{(q, \infty)},$ для которых
$d=\alpha \beta,$ т. е. $d_k=\alpha_k\beta_k$ для $k=1,2,\dots.$ Здесь
$l^0_{(q,\infty)}:=\{(\beta_k):\ \exists a_k\to0, |\beta_k|\le {a_k}/{k^{1/q}}\}.$
Обратно, если $\alpha:= (\alpha_k)\in l_1$ и $\beta:= (\beta_k)\in l^0_{(q, \infty)},$ то
$\alpha \beta\in l_{(s,1)}.$
Более того, $l_1\cdot l_{(q, \infty)}=l_{(s,1)}.$
\end{lem}
\begin{proof}
Возьмем $d\in l_{(s,1)}$ (предполагая, что $d=d^*=(d^*_k)$). Тогда
$\sum_{k=1}^\infty k^{1/s}\, d^*_k/k <\infty,$ т. е. $\sum_{k=1}^\infty k^{1/q}\, d^*_k <\infty.$
Пусть $\varepsilon=(\varepsilon_k)$ --- числовая последовательность такая, что $\varepsilon_k\searrow 0$ и
$\sum_{k=1}^\infty d^*_k k^{1/q}/\varepsilon_k<\infty.$ Положим
$\alpha_k:= {d^*_k k^{1/q}}/{\varepsilon_k}, \,  \ \beta_k:= {\varepsilon_k}/{k^{1/q}}.$
Тогда $\alpha:= (\alpha_k)\in l_1$ и $\beta:= (\beta_k)\in l^0_{(q, \infty)}.$ Таким образом,
$d=\alpha \beta\in l_1\cdot l_{(q, \infty)}^0.$
По поводу последних двух утверждений, см. \cite[2.1.13]{PiEig}.
\end{proof}
 
\begin{prop}
Пусть $\alpha\in [0,1/2]$ и  $1/s=1+\alpha.$ Для банахова пространства $Y,$ предположим, что
 
 $(\alpha)$\ существует такая постоянная $C>0$
что для каждого $\varepsilon>0,$ для любого натурального $n$
и для всякого $n$-мерного подпространства $E$ пространства $Y$
существует конечномерный оператор $R$ в $Y$ такой, что
$||R||\le Cn^{\alpha}$ and $||R|_E- {id}_E||_{L(E,Y)}\le \varepsilon.$
 
Тогда $Y\in AP_{(s,1)}.$
\end{prop}
\begin{proof}

Пусть $0\neq z\in Y^*\widehat\otimes_{(s,1)} X.$ Воспользуемся леммой \ref{Ls1}: возьмем представление
$z=\sum_{k=1}^\infty a_kb_k\, y'_k\otimes x_k,$ в котором 
$(x_k), (y'_k)$ ограничены, $(a_k)\in l_1,$ $(b_k)\in l^0_{q\infty}$ и $b_k\searrow0.$
Тогда $(\widetilde x_k:=b_kx_k)\in l^0_{q \infty}(X)$ и, для достаточно малого $\varepsilon>0$
(которое будет выбрано ниже), можно найти оператор $R\in X^*\otimes X$ с тем свойством, что
$\sup_n ||R\widetilde x_n-\widetilde x_n||\le \varepsilon$ (здесь мы использовали свойства рассматриваемого 
пространства $X,$ отмеченные в разделе 1 работы \cite{Trend}).
Так как $z\neq0,$ то можно найти оператор $V\in L(Y^*, X^*)$ такой, что
$\sum_{k=1}^\infty a_k\, \langle Vy'_k, \widetilde x_k\rangle=1.$
Теперь, когда оператор $V$ выбран, получаем:
\begin{equation}\nonumber
 1= \sum_{k=1}^\infty a_k\, \langle Vy'_k, \widetilde x_k-R\widetilde x_k\rangle +\sum_{k=1}^\infty a_k\, \langle Vy'_k, R\widetilde x_k\rangle
\end{equation}
\begin{equation}\nonumber
\le \varepsilon\, ||(a_k)||_{l_1}\,  ||V||\cdot const + |\sum_{k=1}^\infty a_kb_k\, \langle R^*Vy'_k,  x_k\rangle|,
\end{equation}
 и, если $\varepsilon$ достаточно мало, для конечномерного оператора $R^*V: Y^*\to X^*$ имеем:
\begin{equation}\nonumber
 |\tr z^t\circ (R^*V)|= |\tr (R^*V)\circ z^t|= |\sum_{k=1}^\infty a_kb_k\, \langle R^*Vy'_k, x_k\rangle|>0.
\end{equation}
 Последняя сумма есть ядерный след тензорного элемента $\sum_{k=1}^\infty a_kb_k\, R^*Vy'_k\otimes x_k,$
 который является композицией $R\circ z_0$ конечномерного оператора $R$ и тензорного элемента
 $z_0:=\sum_{k=1}^\infty a_kb_k\, Vy'_k\otimes x_k,$ принадлежащего тензорному произведению
 $X^*\widehat\otimes_{(s,1)} X$ по второй части леммы \ref{Ls1}. Отсюда следует, что как $z_0,$
 так и $z$ порождают ненулевые операторы $\widetilde z_0$ и $\widetilde z.$
 \end{proof}
 
Свойствами $(\alpha)$ обладают, в частности, факторпространства подпространств и
подпространства факторпространств пространств $L_p$ (при $1\le p\le\infty$ с $\alpha=|1/2-1/p|$);
см. обсуждение этого в разделе 1 статьи \cite{Trend}, а также в \cite[Предложение 9]{RQ14}.
Поэтому, в частности, получаем
 
 \begin{cor}
 Пусть $s\in (0,1],$ $p\in [1,\infty]$ и $1/s=1+|1/p-1/2|.$
 Если банахово пространство $Y$ изоморфно подпространству факторпространства
 (или факторпространству подпространства) некоторого $L_p$-пространства, то
 оно обладает свойством $AP_{(s,1)}$ (и, следовательно, свойством $AP_{s}).$
 \end{cor}

Кстати, при $s=2/3$ получаем уже упоминавшийся результат о наличии свойства $AP_{(2/3,1)}$
у любого банахова пространства.

H. K\"onig \cite[2.c.9]{Kon86} показал. что если $1<p<\infty$ и $0<s<1, 1/r=1/s-|1/p-1/2|,$ то для $X=L_p(\mu),$ то
 собственные значения любого оператора $T\in N_{s}(X)$ лежат в $l_{(r,s)}.$
 Поэтому,
 получаем небольшое усиление ранее полученных теорем (Латиф-Рейнов, 2013-2016) о
 ядерных операторах в подпространствах факторпространств пространств $L_p(\mu):$
  
 Итак, обобщение предложения \cite[2.c.9]{Kon86}:
 
\begin{prop} \label{P5}
Пусть $1<p<\infty$ и $0<s<1, 1/r=1/s-|1/p-1/2|.$ 
Существует такая постоянная $C_{s,p}>0,$ что для всякого подпространства $X$ 
любого факторпространства  пространства $L_p(\mu)$
и для любого оператора $T\in N_{s}(X)$
\begin{equation}\nonumber
||\{\mu_k(T)\}||_{l_{(r,s)}}\le C_{s,p}||T||_{N_{s}}.
\end{equation}
(здесь $\{\mu_k(T)\}$ --- полный набор собственных значений оператора $T).$ 
При $r=1$ и $1=1/s - |1/p-1/2|$ 
полный набор собственных значений оператора $T$ абсолютно суммируем, 
для любого оператора $T\in N_s(X)$ его ядерный след $\operatorname{trace}\, T$ вполне определен
и совпадает с его спектральным следом, т. е.
\begin{equation}\nonumber
\operatorname{trace}\, T=\sum_{k=1}^\infty \mu_k(T).
\end{equation}
 \end{prop}
 \begin{proof}
Любое банахово пространство имеет как свойство $||\cdot||_s$-продолжения, так и свойство
$||\cdot||_s$-лифтинга (см. пример \ref{E3.4}).  Из \cite[2.c.9]{Kon86} следует, что $L_p$-пространства имеют
спектральный тип $l_{(r,s)}.$ По теоремам \ref{T5.1} и \ref{T5.2}, как подпространства, так и факторпространства
$L_p$-пространств имеют спектральный тип $l_{(r,s)}.$ По тем же теоремам их, соответственно,
факторпространства и подпространства также имеют имеют спектральный тип $l_{(r,s)}.$
Применим предложение \ref{P3} к семействам факторпространств подпространств и подпространств
факторпространств пространств $L_p,$ рассматривая прямые суммы по типу $l_p.$ Получим
нужные нам неравенства. Применяя в этих же ситуациях предложение \ref{P4} (учитывая наличие
свойств $AP_s),$ получаем формулы следа.
\end{proof}
 
 %%%%%%%%%%%%%%%%%%%%%%%%%%%% 

\subsection{Операторный идеал $N_{(r,s),p}$}
Пусть $0<r,s\le1, 1\le p\le2.$ Определим новую проективную квазинорму $||\cdot||_{(r,s),p}$
следующим образом. Если $u\in X\widehat\otimes Y,$ то
\begin{equation}\nonumber
\|u\|_{{(r,s),p}}:=\inf\left\{\|(\lambda_i)_{i=1}^{\infty}\|_{l_{(r,s)}} \|(x_{i})_{i=1}^{\infty}\|_{l_{\infty}(X)} \cdot\|(y_{i})_{i=1}^{\infty}\|_{l_{p'}^{w}(X)}:\ u=\sum_{i=1}^{\infty} \lambda_i x_{i}\otimes y_{i}\right\}
\end{equation}

Получаем новое тензорное произведение $X\widehat\otimes_{(r,s),p} Y,$ состоящее из тензорных элементов
$u\in X\widehat\otimes Y$ конечной квазинормы $||\cdot||_{(r,s),p}.$ Оно квази-банахово (проверяется
стандартным образом на абсолютно сходящихся рядах) и является частичным обобщением
тензорного произведения Лапресте \cite{166}. 

Естественным образом мы приходим к квазинормированному операторному идеалу 
$N_{(r,s),p},$ рассматривая фактор-отображения $X^*\widehat\otimes_{(r,s),p} Y\to N_{(r,s),p}(X,Y).$
Всякий оператор из этого идеала допускает соответствующее разложение в ряд.
Мы применим полученные выше факты (аналогично случаю операторов в подпространствах
$L_p$-пространств) к операторам из этого нового операторного идеала типа Лапресте (идеала
Лоренца-Лапресте).

\begin{rem}
Более общим является тензорное произведение $X\widehat\otimes_{(r,s),p,q} Y,$ получаемое
аналогичным образом, но с дополнительным ограничением на последовательность $(y_{i}):$
требуется, чтобы эта последовательность была слабо $q$-суммируемой, где $1\le q<\infty.$
Мы не рассматриваем его здесь в силу ограниченности объема работы.
\end{rem}

%%%%%%%%%%%%%%%%%%%%%%%%%%%
Перейдем теперь к операторам из $N_{(r,s),p}.$

\begin{prop}
Если $1\le p\le2, 1/r=1/p+1/2,$ то  всякое банахово пространство обладает свойством $AP_{(r,1),p}.$ 
\end{prop}
 
Отметим, что при $p=1$ свойство $AP_{(r,1),p}$ превращается в $AP_{(2/3,1)},$ о наличия которого в любом банаховом пространстве известно из
работ \cite{Trend}   и \cite{HP10}.
  
\begin{cor}
Если $0<r\le1, 1/r=1/p+1/2, 1\le p\le2$ и $0<s\le1,$ то $N_{(r,s),p}\subset N_{(r,1),p}$
и, следовательно,  всякое банахово пространство обладает свойством $AP_{(r,s),p}.$
\end{cor}

 \begin{prop}
Идеал $N_{(r,s),p}$ имеет спектральный тип $l_{(1,s)}.$
\end{prop}

 \begin{proof}
Мы приведем доказательство, в ходе которого будут получены все три сформулированные выше утверждения.
 Пусть $0<r\le1, 1/r=1/p+1/2, 1\le p\le2$ и $0<s\le1.$ Предположим, что
$X\notin AP_{(r,s),p}.$
Пусть $z\in X^*\widehat\otimes_{(r,s),p} X$ --- такой элемент, что
 $\operatorname{trace}\, z=1, \tilde z=0.$
 
\begin{equation}\nonumber
 z=\sum_{k=1}^\infty \lambda_k x'_k\otimes x_k,
\end{equation}
 где $(\lambda_k)\in l_{r,s},$ $(x'_k)\in l_\infty(X^*) $ et $ (x_k) $ --- слабо $ p'$-суммируема.
 
 Поскольку $z=\sum \lambda_k x'_k\otimes x_k,$ где $(\lambda_k)\in l_{r,s},$ $(x'_k)\in l_\infty(X^*) $ et $ (x_k) $ --- слабо $ p'$-суммируема, то $ \tilde z $ может быть факторизован как
\begin{equation}\nonumber
 \tilde z:\ X\overset{A}\to l_\infty \overset{\Delta}\to l_1 \overset{j}\to l_p \overset{V}\to X,
\end{equation}
 где $Ax=\{\langle x'_k,x\rangle\}\in l_\infty$ для $x\in X,$ $V\{\delta_k\}:= \sum \delta_k x_k$ для 
 $\{\delta_k\}\in l_p,$
 $ j $ --- вложение, $ \Delta $ --- диагональный оператор с диагональю $(\lambda_k)$ из $ l_{(r,s)}. $
 Так как $\tilde z=0,$ то $V|_{j\Delta A(X)}=0.$ Рассмотрим $S:= j\Delta AV: l_p\to l_p:$
 \begin{equation}\nonumber
 S\{\delta_k\}= \sum_{k=1}^\infty \delta_k j\Delta Ax_k= \sum_{k=1}^\infty \delta_k  j\Delta(\sum_{j=1}^\infty \langle x'_j,x_k\rangle e_j)=
 \sum_{k,j=1}^\infty \lambda_j \langle x'_j, \delta_k x_k\rangle e_j = 
 \end{equation}
\begin{equation}\nonumber
\sum_{j=1}^\infty \lambda_j \langle x'_j, \sum_{k=1}^\infty \delta_k x_k\rangle e_j  
 = \sum_{j=1}^\infty \lambda_j \langle \{\delta_k\}_k, \{\langle x'_j,x_k\rangle\}_k\rangle e_j  
 \in l_p
\end{equation}
 для $\{\delta_k\}\in l_p,$    
 Положим $\langle \{\delta_k\}_k, \{\langle x'_j,x_k\rangle\}_k\rangle=: \psi_j(\delta).$ Тогда 
 $ S\{\delta_k\}=  \sum_{j=1}^\infty \lambda_j \psi_j(\delta) e_j.$  Следовательно,
\begin{equation}\nonumber
\tr S= \sum_j \lambda_j \psi_j(e_j)= \sum_j \lambda_j \langle x'_j, x_j\rangle=1.
\end{equation} 
Очевидно, $S^2=0$ и $\operatorname{trace}\, S=\operatorname{trace}\, z=1.$
  
Рассмотрим диагональный оператор $j\Delta: l_\infty\to l_p$ с диагональю из $ l_{(r,s)}.$
Из \cite[2.9.17*]{PiEig} следует, что этот оператор
есть оператор вейлевского (а значит и спектрального) типа $(1,s)$ (\cite[3.6.2*]{PiEig}; подробно об операторах Вейля
см. указанную монографию). Следовательно,  идеал $N_{(r,s),p}$ имеет спектральный тип $l_{1,s}.$
 
Поскольку $S\in N_{(r,1),p}(l_p,l_p):$  
\begin{equation}\nonumber
 S:   l_p \overset{V}\to X \overset{A}\to l_\infty \overset{\Delta}\to l_1 \overset{j}\to l_p,
\end{equation}
 $N_{(r,1),p}$ имеет спектральный тип $l_{1}$
и пространство $l_p$ обладает свойством аппроксимации Гротендика (а, значит, и свойством
$AP_{(r,1),p}),$  то ядерный след $\operatorname{trace}\, S $
вполне определен и равен сумме всех собственных значений оператора $S$
(по предложению \ref{P4}, в котором сейчас $\mathcal F$ есть семейство всех банаховых пространств).
Противоречие с тем, что $S^2=0.$
\end{proof}
  
Применяя предложения \ref{P3} и \ref{P4} для рассматриваемой ситуации, получаем:
  
\begin{theor}
  Пусть $0<r\le1, 1/r=1/p+1/2,  1\le p\le2$ и $0<s\le1.$
   Существует такая постоянная $C>0,$ что для всякого банахова пространства $X$
  и для любого оператора $T\in N_{(r,s),p}(X)$
\begin{equation}\nonumber||\{\mu_k(T)\}||_{l_{(1,s)}}\le C||T||_{N_{(r,s),p}}
\end{equation}
  (здесь $\{\mu_k(T)\}$ --- полный набор собственных значений оператора $T).$
  В частности, %при $r=1$ и $1=1/s - |1/p-1/2|$ 
  полный набор собственных значений оператора $T$ абсолютно суммируем, 
 его ядерный след $\operatorname{trace}\, T$ вполне определен
  и совпадает с его спектральным следом, т. е.
\begin{equation}\nonumber\operatorname{trace}\, T=\sum_{k=1}^\infty \mu_k(T).
\end{equation}
 \end{theor}

Отметим {\it частные случаи  теоремы для $N_{(r,s),p}:$}
  
Пусть  $0<r\le1, 1/r=1/p+1/2, 1\le p\le2$ и $0<s\le1.$
  \bigskip
  
  a)\,
  $r=1, s=1, p=2:$\ В.Б. Лидский (1959), А. Пич (1980)
   \medskip
   
  b)\,
  $r=2/3, s=2/3, p=1:$\ А. Гротендик (1955)
   \medskip
   
  c)\,
  $r=2/3, s=1, p=1:$\ А. Хинрихс и А. Пич (2010) и, независимо, автор (2016)
   \medskip
   
  d)\,
  $0\le r\le1, s=r, 1/r=1/2+1/p:$\  О. Рейнов и К. Латиф (2013)
 \medskip
 
 Теорема соединяет в одной шкале операторов частные случаи c) и a):
 
\begin{equation}\nonumber
 \{r=2/3, s=1, p=1\} \longrightarrow \{2/3\le r\le1, s=1, 1/r=1/p+1/2\}
 \end{equation}
\begin{equation}\nonumber\longrightarrow \{r=1, s=1, p=2\}
\end{equation}

{\it  О точности результатов}
 
 Все результаты, приведенные  до теоремы об $N_{(r,s),p},$ точны. Теорема точна для случаев, когда $r=s.$
 Для $r\neq s$ проблема возникает уже в частном случае $N_{(2/3,1)}$
 (т. е. при $p=1).$ 
  
 Из статьи А. Хинрихса и А. Пича \cite{HP10} (2010) в нашей формулировке:
 \medskip

 Верно ли что в шкале пространств Лоренца $l_{r,s}$
 результат <<{\it любое банахово пространство обладает свойством} $AP_{(2/3,1)}$>> 
 есть наилучший результат?

 %%%

 \end{document}